\documentclass[12pt]{amsart}
\usepackage[margin=1in]{geometry}

\usepackage{aliascnt}
\usepackage{amssymb}
\usepackage{amsmath}
\usepackage{enumitem}
\usepackage{xcolor}
\definecolor{darkblue}{rgb}{0.0, 0.0, 0.8}
\usepackage[colorlinks=true, allcolors=darkblue]{hyperref}
\usepackage{todonotes}
\usepackage{tabularx}
\usepackage{mathtools}
\DeclareMathOperator{\arccot}{arccot}
\usepackage{autonum}
\usepackage{comment}

\usepackage{ragged2e}

\usepackage{dsfont}

\newtheorem{theorem}{Theorem}[section]

\newaliascnt{lemma}{theorem}
\newtheorem{lemma}[lemma]{Lemma}
\aliascntresetthe{lemma}

\newaliascnt{proposition}{theorem}
\newtheorem{proposition}[proposition]{Proposition}
\aliascntresetthe{proposition}

\newaliascnt{corollary}{theorem}

\aliascntresetthe{corollary}

\newaliascnt{conjecture}{theorem}

\aliascntresetthe{conjecture}

\newaliascnt{assumption}{theorem}

\aliascntresetthe{assumption}

\newaliascnt{definition}{theorem}
\newtheorem{definition}[definition]{Definition}
\aliascntresetthe{definition}

\newaliascnt{question}{theorem}
\newtheorem{question}[question]{Question}
\aliascntresetthe{question}

\newaliascnt{remark}{theorem}
\newtheorem{remark}[remark]{Remark}
\aliascntresetthe{remark}

\newaliascnt{example}{theorem}

\aliascntresetthe{example}

\newtheorem*{notation*}{Notation}
\newtheorem*{theorem*}{Theorem}
\newtheorem*{conjecture*}{Conjecture}

\newaliascnt{clai}{theorem}
\newtheorem{clai}[clai]{Claim}
\aliascntresetthe{clai}

\numberwithin{figure}{section}
\numberwithin{table}{section}
\numberwithin{equation}{section}

\allowdisplaybreaks

\newcommand{\Diff}{\mathrm{Diff}}

\usepackage[capitalise,nameinlink]{cleveref}
\crefname{theorem}{theorem}{theorems}
\Crefname{theorem}{Theorem}{Theorems}
\crefname{lemma}{lemma}{lemmas}
\Crefname{lemma}{Lemma}{Lemmas}
\crefname{proposition}{proposition}{propositions}
\Crefname{proposition}{Proposition}{Propositions}
\crefname{corollary}{corollary}{corollaries}
\Crefname{corollary}{Corollary}{Corollaries}
\crefname{conjecture}{conjecture}{conjectures}
\Crefname{conjecture}{Conjecture}{Conjectures}
\crefname{assumption}{assumption}{assumptions}
\Crefname{assumption}{Assumption}{Assumptions}
\crefname{definition}{definition}{definitions}
\Crefname{definition}{Definition}{Definitions}
\crefname{question}{question}{questions}
\Crefname{question}{Question}{Questions}
\crefname{remark}{remark}{remarks}
\Crefname{remark}{Remark}{Remarks}
\crefname{example}{example}{examples}
\Crefname{example}{Example}{Examples}
\crefname{clai}{claim}{claims}
\Crefname{clai}{Claim}{Claims}

\newcommand{\RNum}[1]{\uppercase\expandafter{\romannumeral #1\relax}}

\newcommand{\MAC}{\mathcal{M}_{AC}^0}

\newcommand{\id}{\mathrm{id}}

\renewcommand{\SS}{\mathbb{S}}
\renewcommand{\S}{\mathbb{S}^{2n+1}}

\newcommand{\SiL}{\mathbb{S}_{L^2}^{\infty}}

\newcommand{\RR}{\mathbb{R}}
\newcommand{\GH}{\mathcal{G}^H}
\newcommand{\GL}{\mathcal{G}^L}

\def\CC{{\mathbb C}}
\def\RR{{\mathbb R}}

\def\ZZ{{\mathbb Z}}
\def\d{{\mathrm{d}}}

\def\MAC{\mathcal{M}_{\dot{H}^1}}
\def\GH{\mathcal{G}^{\dot{H}^1}}
\def\GL{\mathcal{G}^{L^2}}

\def\id{\mathrm{id}}

\def\SS{{\mathbb{S}}}
\def\S{{\mathbb{S}^{2n+1}}}

\def\SiL{{\mathbb{S}_{L^2}^{\infty}}}

\def\i{{\mathrm{i}}}

\def\Re{{\mathrm{Re}}}
\def\Im{{\mathrm{Im}}}

\def\Diff{\mathrm{Diff}}

\def\vardalp{\varPhi^{*}\d\alpha}

\setlist[itemize]{noitemsep, topsep=0pt}

\begin{document}

\title[Global weak solutions of the (M2HS)]{On Global Weak Solutions for the Magnetic Two-Component Hunter--Saxton System}

\author{Levin Maier}
\address{Faculty of Mathematics and Computer Science, University of Heidelberg}
\email{lmaier@mathi.uni-heidelberg.de}

\begin{abstract}
We study the magnetic two-component Hunter--Saxton system \eqref{eq:M2HS}, which was recently derived in \cite{M24} as a magnetic geodesic equation on an infinite-dimensional configuration space. While the geometric framework and the global weak flow were outlined there, the present paper provides the analytical foundations of this construction from the PDE perspective.
First, we derive an explicit solution formula in Lagrangian variables via a Riccati reduction, yielding an alternative proof of the blow-up criterion together with an explicit expression for the blow-up time. Second, we rigorously construct global conservative weak solutions by developing the analytic theory of the relaxed configuration space and the associated weak magnetic geodesic flow, thereby realizing the geometric program proposed in \cite{M24}.
\end{abstract}

\maketitle

\section{Introduction}
\subsection{The equations and motivation}\label{sec:M2HS-def}
In this paper, we study the magnetic two-component Hunter--Saxton system recently introduced in \cite{M24}:
\begin{equation}\label{eq:M2HS}
\left\{
\begin{aligned}
u_{tx} &= -\frac12 u_x^2 - u\,u_{xx} + \frac12\rho^2 - \bigl(s\rho + 2(c^2 - s\delta)\bigr)
&&\text{in }\SS^1\times(0,\infty),\\
\rho_t &= -(\rho u)_x + s\,u_x
&&\text{in }\SS^1\times(0,\infty),\\
u(\cdot,0) &= u_0,\qquad \rho(\cdot,0)=\rho_0
&&\text{in }\SS^1.
\end{aligned}
\right.
\tag{M2HS}
\end{equation}
Here $s\in\RR$ is the magnetic parameter,
\begin{equation}\label{eq:delta_def}
\delta:=\frac12\int_{\SS^1}\rho(t,x)\,\d x
\end{equation}
denotes the contact angle, and
\begin{equation}\label{eq:c_def}
c^2:=\frac14\int_{\SS^1}\bigl(u_x(t,x)^2+\rho(t,x)^2\bigr)\,\d x
\end{equation}
is the $\dot H^1$-energy. Both quantities are conserved along smooth solutions by \cite[Thm. 5.1]{M24}.

The parameter $s$ can be interpreted as the strength of an underlying magnetic field in the geometric formulation of the system. It couples the velocity and density components and leads to dynamics that differ substantially from the non-magnetic case. A central goal of this paper is to combine an explicit description of smooth solutions, including blow-up behavior, with a global conservative weak theory that continues the dynamics beyond singularity formation.

For $s=0$, the system \eqref{eq:M2HS} reduces to the classical two-component Hunter--Saxton system \cite{wu11,l13}. If $\rho \equiv s$, one recovers the scalar Hunter--Saxton equation, which models the propagation of nonlinear orientation waves in nematic liquid crystals \cite{HunterSaxton1991} and has a rich geometric structure related to geodesic flows \cite{km02,l07.1,l07.2}.

Beyond its liquid-crystal origin, the Hunter--Saxton system also appears in mathematical physics as a model for one-dimensional non-dissipative dark matter (the Gurevich--Zybin system) and as a short-wave limit of the two-component Camassa--Holm system \cite{Pavlov2005,ConstantinIvanov2008,EscherLechtenfeldYin2007,Guo2010,GuoZhou2009}. It belongs to a broader class of coupled third-order systems that includes models for axisymmetric Euler flow with swirl and vorticity dynamics \cite{Wunsch2010HS,HouLi2008,ConstantinLaxMajda1985,OkamotoSakajoWunsch2008,Okamoto2009}.

From a geometric viewpoint, system \eqref{eq:M2HS} was derived in \cite{M24} as the equation of motion of a Hamiltonian system on an infinite-dimensional configuration space. In that setting, smooth solutions can be interpreted as trajectories of a geometric flow, and the formation of singularities corresponds to a degeneration of the associated Lagrangian map. Moreover, \cite{M24} proposed that by enlarging the configuration space, this geometric flow can be continued globally in a weak sense even after smoothness breaks down. The present paper provides the analytical foundations of this program and makes this continuation rigorous at the PDE level.

With this motivation in place, we now turn to the main contributions of this work.
\subsection{Main contributions of this work}

The principal results of this article can be summarized as follows:
\begin{itemize}
    \item \textbf{Weak magnetic geodesic framework on $\MAC$.}
    We introduce a weak notion of magnetic geodesic on the relaxed configuration space $\MAC$ and develop the analytical framework required to study weak magnetic flows (see \Cref{def:weak_magnetic_geodesic}).

    \item \textbf{Global existence of the weak magnetic flow.}
    We prove global existence of the weak magnetic flow on $\MAC$ and provide an explicit construction of global unit-speed weak magnetic geodesics for admissible initial data (see \Cref{thm:global_weak_magnetic_flow}).  
    As special cases, this recovers the known weak flows for the two-component Hunter--Saxton system ($s=0$) \cite[Thm.~4.1]{wu11} and for the scalar Hunter--Saxton equation ($\rho\equiv s$) \cite[Thm.~4.1]{l07.2}.

    \item \textbf{Global conservative weak solutions of (M2HS).}
    By transferring the weak magnetic flow back to Eulerian variables, we obtain global conservative weak solutions of the magnetic two-component Hunter--Saxton system \eqref{eq:M2HS} (see \Cref{thm:existence_global_weak_sol_M2HS}).  
    This extends the classical weak solution theories for the cases $s=0$ \cite[Thm.~4.2]{wu11} and $\rho\equiv s$ \cite[Thm.~4.2]{l07.2}.

    \item \textbf{Conservation laws.}
    We prove that both the $\dot H^1$--energy and the contact angle are conserved almost everywhere in time along the weak solutions constructed here (see \Cref{thm:existence_global_weak_sol_M2HS}).

    \item \textbf{Explicit solution formula and blow-up characterization.}
    We derive an explicit representation formula for smooth solutions of \eqref{eq:M2HS} via a Riccati reduction in Lagrangian variables (see \Cref{thm:explicit_formula}).  
    This yields an alternative proof of the geometric blow-up criterion from \cite{M24} together with an explicit expression for the blow-up time.  
    In the special case $s=0$, Wunsch’s explicit solution formula for the two-component Hunter--Saxton system \cite[Thm.~2.1]{wu11} is recovered.
\end{itemize}

Having summarized the main results, we now briefly describe the structure of the paper.
\subsection{Outline of the article}
In \Cref{sec: prelim} we review the geometric background on magnetic geodesics and explain how the magnetic two-component Hunter--Saxton system arises as a magnetic geodesic equation on the configuration space.

In \Cref{sec:explicit_solution_formula} we pass to Lagrangian variables and reduce the system to a complex Riccati equation. This leads to an explicit solution formula, a blow-up criterion, and an explicit blow-up time; we also discuss the regime of large magnetic parameter $s$.

The weak theory is developed in \Cref{sec:global_weak}. There we introduce the relaxed configuration space $\MAC$, define weak magnetic geodesics, prove global existence of the weak magnetic flow, and use it to construct global conservative weak solutions of \eqref{eq:M2HS} with conservation of energy and contact angle almost everywhere in time.

Technical lemmata used in the weak theory are collected in \Cref{Appendix global weak solt: techni lemmas}.
\\

\paragraph{\textbf{Acknowledgments:}}The author acknowledge funding from the Deutsche Forschungsgemeinschaft (DFG, German Research Foundation) – 281869850 (RTG 2229), 390900948 (EXC-2181/1), and 281071066 (TRR 191). 


\section{Preliminaries: geometric origin of \texorpdfstring{\eqref{eq:M2HS}}{(M2HS)}}\label{sec: prelim}

In this section we recall the geometric framework of magnetic geodesics and explain how the magnetic two-component Hunter--Saxton system arises as a magnetic geodesic equation on an infinite-dimensional configuration space. We also present an equivalent formulation of the equation that will be fundamental for the weak theory developed later, and relate the system to a linear ODE on a Hilbert sphere.

\subsection{Background: Magnetic geodesics}

We present the mathematical framework used to study the dynamics of a charged particle in the presence of a magnetic field, following V.~Arnold's pioneering approach~\cite{ar61}.
    
Let $(M,g)$ be a strong Riemannian Hilbert manifold in the sense of \cite{La99} and let $\sigma\in\Omega^2(M)$ be a closed two-form. The form $\sigma$ is called a \emph{magnetic field}, and $(M,g,\sigma)$ is called a \emph{magnetic system}. This triple determines a skew-symmetric bundle endomorphism $Y\colon TM\to TM$, the \emph{Lorentz force}, defined by
\begin{equation}\label{e:Lorentz}
    g_q\bigl(Y_qu,v\bigr)=\sigma_q(u,v),\qquad \forall\, q\in M,\ \forall\,u,v\in T_qM.
\end{equation}

A smooth curve $\gamma\colon \RR\to M$ is called a \emph{magnetic geodesic} of $(M,g,\sigma)$ if it satisfies
\begin{equation}\label{e:mg}
		\nabla_{\dot\gamma}\dot\gamma = Y_{\gamma}\dot\gamma,
\end{equation}
where $\nabla$ denotes the Levi-Civita connection of the metric $g$. Equation~\eqref{e:mg} reduces to the geodesic equation \(\nabla_{\dot{\gamma}} \dot{\gamma} = 0\) when \(\sigma = 0\), i.e.\ when the magnetic field vanishes.

Like standard geodesics, magnetic geodesics have constant kinetic energy
\[
E(\gamma,\dot\gamma):=\tfrac12 g_\gamma(\dot\gamma,\dot\gamma),
\]
and therefore travel at constant speed $|\dot\gamma|:=\sqrt{g_\gamma(\dot\gamma,\dot\gamma)}$, since the Lorentz force $Y$ is skew-symmetric.

This conservation of energy reflects the Hamiltonian nature of the system. Indeed, the \emph{magnetic geodesic flow} is defined on the tangent bundle by
\[
\varPhi_{g,\sigma}^t\colon TM\to TM,\quad (q,v)\mapsto \bigl(\gamma_{q,v}(t),\dot\gamma_{q,v}(t)\bigr),\quad \forall t\in\RR,
\]
where $\gamma_{q,v}$ is the unique magnetic geodesic with initial condition $(q,v)\in TM$. 

A key difference from standard geodesics is that magnetic geodesics with different speeds are not simply reparametrizations of unit-speed magnetic geodesics. This can be seen, for instance, from the fact that the left-hand side of \eqref{e:mg} scales quadratically with the speed, while the right-hand side scales only linearly. Therefore, an important aspect of the theory is to understand the similarities and differences between standard and magnetic geodesics as the kinetic energy varies.

\subsection{\eqref{eq:M2HS} as magnetic geodesic equation}

We now introduce the geometric setting in more detail. Let \(\SS^1 = \RR / \ZZ\) denote the unit circle, and let \(\mathrm{Diff}_0^m(\SS^1)\) be the half-Lie group of Sobolev diffeomorphisms of class \(H^m\) with smooth right multiplication and only continuous left multiplication; see \cite{Bauer_2025} for the notion of a half-Lie group. This group consists of all diffeomorphisms of \(\SS^1\) of Sobolev class \(H^m\) that fix a designated point. Unless stated otherwise, we assume \(m > \frac{5}{2}\). Note that \(\mathrm{Diff}_0^m(\SS^1) \cong \mathrm{Dens}^m(\SS^1)\), where \(\mathrm{Dens}^m(\SS^1)\) denotes the space of probability densities on \(\SS^1\) of Sobolev class \(H^m\); see \cite[§5]{EM70} and \cite{km02}. 

We further denote by \(\SS^1_{4\pi}\) the circle of length \(4\pi\), and by \(H^m(\SS^1, \SS^1_{4\pi})\) the space of maps of Sobolev class \(H^m\). The half-Lie group \(G^m\) is defined as 
\begin{equation}\label{e: defi of Gm}
    G^m = \mathrm{Diff}_0^m(\SS^1) \rtimes H^{m-1}(\SS^1, \SS^1_{4\pi}),
\end{equation}
where the explicit group product structure is described in \cite[§2]{Lennels13}. This group carries a right-invariant metric, the $\dot{H}^1$-metric, defined at the identity \((\id, 0)\) by
\begin{equation}\label{e: defi H1 dor metric at identity}
\mathcal{G}^{\dot{H}^1}_{(\id, 0)}\bigl((u, \rho), (v, \psi)\bigr)
:= \langle (u, \rho), (v, \psi)\rangle_{(\id, 0)}^{\dot{H}^1}
= \frac{1}{4} \int_{\SS^1} u_x v_x + \rho \psi \, \mathrm{d}x,
\end{equation}
where the tangent space at the identity is \(T_{(\id, 0)}G^m \cong H^m_0(\SS^1, \RR) \times H^{m-1}(\SS^1, \RR)\). Here \(H^m_0(\SS^1, \RR)\) denotes the space of Sobolev functions with zero mean. By right invariance, the metric in \eqref{e: defi H1 dor metric at identity} extends to all points \((\varphi, \tau) \in G^m\) as 
\begin{equation}\label{e: definition H1 dot metric at arbitrary element}
	\mathcal{G}^{\dot{H}^1}_{(\varphi, \tau)}\bigl((U_1, U_2), (V_1, V_2) \bigr)
:= \langle (U_1, U_2), (V_1, V_2)\rangle_{(\varphi, \tau)}^{\dot{H}^1}
= \frac{1}{4} \int_{\SS^1} \frac{U_{1x}V_{1x}}{\varphi_x} + U_2 V_2 \varphi_x \, \mathrm{d}x,
\end{equation}
where \((U_1, U_2), (V_1, V_2) \in T_{(\varphi, \tau)}G^m \cong H^m_0(\SS^1, \RR) \times H^{m-1}(\SS^1, \RR)\). For further details we again refer to \cite[§2]{Lennels13}. With that notation fixed we are in place to state: 

\begin{theorem}[{\cite[Thm. 5.1]{M24}}]
\label{t: magnetic Hunter Saxton system}
A \(C^2\)-curve \((\varphi,\tau): [0,T)\longrightarrow G^m\), where \(T>0\) is the maximal existence time, is a magnetic geodesic of the system \(\left(G^{m}, \GH, \varPhi^{*}\mathrm{d}\alpha\right)\) with \(m > \frac{5}{2}\) if and only if  
\[
(u = \varphi_t \circ \varphi^{-1}, \rho = \tau_t \circ \varphi^{-1}) \in C\left([0,T), T_{(\id,0)}G^m \right) \cap C^1\left([0,T), T_{(\id,0)}G^{m-1} \right)
\]
is a solution of the magnetic two-component Hunter--Saxton system \eqref{eq:M2HS}, that is:
\begin{align}
\begin{cases}
u_{tx} = -\frac{1}{2} u_x^2 - u\, u_{xx} + \frac{1}{2} \rho^2 - (s\rho + 2(c^2 - s\delta)), \\
\rho_t = -(\rho u)_x + s u_x,
\end{cases} \tag{M2HS}
\end{align}
where \(c^2\), the \(\dot{H}^1\)-energy of the system, is a conserved quantity given by
\[
c^2 
= \frac{1}{4} \int_{\mathbb{S}^1} \bigl(u_x^2 + \rho^2\bigr) \, \mathrm{d}x.
\]
Additionally, the contact angle \(\delta\) is another conserved quantity, given by
\[
\delta=\frac{1}{2} \int_{\SS^1} \rho \, \mathrm{d}x.
\]
\end{theorem}

The following reformulation of the magnetic geodesic equation will be crucial later on, both for defining global weak solutions and for constructing them explicitly.

\begin{proposition}\label{Prop: magnetic geodesic equation in G}
The magnetic geodesic equation of \(\left(G^m, \GH, \vardalp \right)\) at the point \((\varphi, \tau) \in G^m\) is given by:
\begin{align}
\begin{pmatrix}
\left(u_t + u u_x \right)\circ \varphi - \frac{1}{2} \left( \int_0^{\varphi(\cdot)} \big(u_x^2(y) + \rho^2(y)\big) \,\d y - \varphi(\cdot)(c^2 - s \delta) \right) - s \int_{0}^{\varphi(\cdot)} \rho(y) \,\d y\\
\left(\rho_t + (u\rho)_x - s u_x \rho \right)\circ \varphi
\end{pmatrix}
= 0,
\end{align}
where \(c^2\) and \(\delta\) are the conserved quantities as in \Cref{t: magnetic Hunter Saxton system}.
\end{proposition}

From a geometric viewpoint, the dynamics of \eqref{eq:M2HS} can be related to a linear ODE on an infinite-dimensional Hilbert sphere.

Combining results from \cite{M24}, one finds that a curve $(\varphi, \tau)$ is a magnetic geodesic of $(G^m, \GH, \varPhi^*\mathrm{d}\alpha)$ with magnetic strength $s$ and speed $c$ if and only if the complex-valued function
\[
\gamma := \sqrt{\varphi_x}\,e^{\i \tau/2}
\]
satisfies
\begin{equation}\label{e:magnetic geodesic equation in SIL}
    \ddot{\gamma} - \i s \dot{\gamma} + \bigl(c^2 - s \delta\bigr)\gamma = 0
    \quad \text{in } L^2(\SS^1,\CC).
\end{equation}
Here \(\delta\) is a conserved quantity representing the angle between \(\dot{\gamma}\) and the vector field \(R(\gamma)=\i\gamma\), and is given by
\begin{equation}\label{e:contact_angle}
\delta = \operatorname{Re} \langle \i\gamma ,\dot{\gamma}\rangle_{L^2}.
\end{equation}

Moreover, by an argument along the lines of \cite[§3]{ABM} the equation~\eqref{e:magnetic geodesic equation in SIL} admits an explicit solution formula. Namely,
\[
\gamma(t)=e^{\i\theta_1 t}p_1 + e^{\i\theta_2 t}p_2,
\]
where
\[
\theta_{1/2} = \frac{s \pm \sqrt{s^2 + 4(c^2 - s\delta)}}{2},
\qquad
p_{1/2} = \mp \frac{\theta_{2/1}\,\gamma(0) + \i \dot{\gamma}(0)}{\theta_1 - \theta_2}.
\]

\section{Explicit solution formula for the magnetic two-component Hunter--Saxton system}\label{sec:explicit_solution_formula}

\subsection*{Overview}
In this section we derive an explicit representation formula for solutions of the magnetic two-component Hunter--Saxton system (M2HS). The formula yields, in particular, an alternative proof of the blow-up criterion established geometrically in \cite{M24}, and it furthermore provides an explicit expression for the blow-up time. Our approach is inspired by the method introduced by Wunsch for the (non-magnetic) two-component Hunter--Saxton system \cite{wu11}, but the presence of the magnetic terms requires additional technical modifications; in particular, all equations appearing below are nontrivial perturbations of those in \cite{wu11}. Setting $s=0$ recovers Wunsch's explicit formula as a special case.

Throughout, we assume $m>\frac52$ and consider a solution pair $(u,\rho)$ of Sobolev class
\[
(u,\rho)\in C\bigl([0,T);H^m(\S^1)\times H^{m-1}(\S^1)\bigr),
\]
so that by Sobolev embedding the functions are at least $C^2$ in space.

\subsection{The (M2HS) system in Lagrangian variables}
Let $(u,\rho)$ solve the magnetic two-component Hunter--Saxton system on $\SS^1$:
\begin{equation}\label{eq:M2HS_paper}
\begin{cases}
u_{tx} = -\frac12 u_x^2 - u\,u_{xx} + \frac12\rho^2 -\bigl(s\rho +2(c^2-s\delta)\bigr),\\[0.2em]
\rho_t = -(\rho u)_x + s u_x,\\[0.2em]
u(0,\cdot)=u_0,\qquad \rho(0,\cdot)=\rho_0.
\end{cases}
\end{equation}
Let $\varphi$ denote the flow generated by $u$, i.e.
\begin{equation}\label{eq:flow_paper}
\varphi_t(t,x) = u(t,\varphi(t,x)),\qquad \varphi(0,x)=x.
\end{equation}
Define the Lagrangian variables
\begin{equation}\label{eq:UP_def_paper}
U(t,x):=u_x\bigl(t,\varphi(t,x)\bigr),\qquad P(t,x):=\rho\bigl(t,\varphi(t,x)\bigr).
\end{equation}
A direct computation using the chain rule yields
\[
\frac{d}{dt}U
= u_{tx}\circ\varphi+(u_{xx}\circ\varphi)\,\varphi_t,
\qquad
\frac{d}{dt}P
= \rho_t\circ\varphi+(\rho_x\circ\varphi)\,\varphi_t.
\]
Substituting \eqref{eq:M2HS_paper} and using $\varphi_t=u\circ\varphi$, one obtains the closed ODE system
\begin{equation}\label{eq:UP_system_paper}
\begin{cases}
U_t = -\frac12U^2+\frac12P^2-\bigl(sP+2(c^2-s\delta)\bigr),\\[0.2em]
P_t = -PU+sU,\\[0.2em]
U(0,\cdot)=(u_0)_x,\qquad P(0,\cdot)=\rho_0.
\end{cases}
\end{equation}

\subsection{Reduction to a Riccati equation}
Introduce the complex-valued function
\begin{equation}\label{eq:Z_def_paper}
Z(t,x):=U(t,x)+iP(t,x).
\end{equation}
Then \eqref{eq:UP_system_paper} is equivalent to the Riccati-type ODE
\begin{equation}\label{eq:Riccati_Z_paper}
\begin{cases}
\displaystyle \frac{d}{dt}Z
= -\frac12 Z^2 + is\,Z -2(c^2-s\delta),\\[0.3em]
Z(0,\cdot)=Z_0:=(u_0)_x+i\rho_0.
\end{cases}
\end{equation}

\begin{lemma}\label{lem:Riccati_reduction}
If $(u,\rho)$ solves \eqref{eq:M2HS_paper} and $\varphi$ is the flow defined by \eqref{eq:flow_paper}, then the function
\[
Z(t,x)=u_x\bigl(t,\varphi(t,x)\bigr)+i\,\rho\bigl(t,\varphi(t,x)\bigr)
\]
solves \eqref{eq:Riccati_Z_paper}.
\end{lemma}

\subsection{Explicit representation formula and blow-up criterion}
We next solve \eqref{eq:Riccati_Z_paper} explicitly. Define
\begin{equation}\label{eq:theta_def_paper}
\theta_{1/2}
=\frac{s\pm\sqrt{s^2+4(c^2-s\delta)}}{2}.
\end{equation}
In the unit speed case $c=1$ and assuming $\delta\neq0$, we have $\theta_1\neq\theta_2$.

\begin{theorem}[Explicit solution formula and blow-up time]\label{thm:explicit_formula}
Let $(u,\rho)$ be a solution of \eqref{eq:M2HS_paper} with initial data
\(
(u_0,\rho_0)\in H^m(\S^1)\times H^{m-1}(\S^1)
\),
assume $c=1$ and $\delta\neq0$, and let $\varphi$ be the flow \eqref{eq:flow_paper}. Then:

\begin{enumerate}[label=(\arabic*)]
\item The quantities $u_x\circ\varphi$ and $\rho\circ\varphi$ admit the explicit representation
\begin{equation}\label{eq:explicit_u_rho_paper}
\begin{aligned}
u_x\bigl(t,\varphi(t,x)\bigr)
&=
\frac{\left(-(\theta_1-\theta_2)\tan\!\left(\frac{\theta_1-\theta_2}{2}t\right)+(u_0)_x(x)\right)\left(1+(u_0)_x(x)\frac{\tan\!\left(\frac{\theta_1-\theta_2}{2}t\right)}{\theta_1-\theta_2}\right)
+(\rho_0(x)-s)^2\frac{\tan\!\left(\frac{\theta_1-\theta_2}{2}t\right)}{\theta_1-\theta_2}}
{\left(1+(u_0)_x(x)\frac{\tan\!\left(\frac{\theta_1-\theta_2}{2}t\right)}{\theta_1-\theta_2}\right)^2
+\left((\rho_0(x)-s)\frac{\tan\!\left(\frac{\theta_1-\theta_2}{2}t\right)}{\theta_1-\theta_2}\right)^2},\\[0.6em]
\rho\bigl(t,\varphi(t,x)\bigr)
&=
\frac{\bigl(1+\tan^2(\frac{\theta_1-\theta_2}{2}t)\bigr)\,(\rho_0(x)-s)}
{\left(1+(u_0)_x(x)\frac{\tan\!\left(\frac{\theta_1-\theta_2}{2}t\right)}{\theta_1-\theta_2}\right)^2
+\left((\rho_0(x)-s)\frac{\tan\!\left(\frac{\theta_1-\theta_2}{2}t\right)}{\theta_1-\theta_2}\right)^2}.
\end{aligned}
\end{equation}

\item The functions $u_x\circ\varphi$ and $\rho\circ\varphi$ develop a singularity if and only if there exist $x_0\in\S^1$ and $t_0>0$ such that
\begin{equation}\label{eq:blowup_condition_paper}
\rho_0(x_0)=s
\qquad\text{and}\qquad
t_0
=\frac{2\,\arccot\!\left(-\frac{(u_0)_x(x_0)}{\theta_1-\theta_2}\right)}{\theta_1-\theta_2}.
\end{equation}
\end{enumerate}
\end{theorem}

\begin{remark}
The condition $\rho_0(x_0)=s$ provides an alternative proof of the blow-up criterion obtained via geometric arguments in \cite{M24}. In contrast to the geometric approach, \eqref{eq:blowup_condition_paper} yields the blow-up time explicitly.
\end{remark}

\begin{remark}
Setting $s=0$ in Theorem~\ref{thm:explicit_formula} recovers \cite[Thm.~2.1]{wu11}.
\end{remark}

\begin{proof}[Proof of Theorem~\ref{thm:explicit_formula}]
Solving \eqref{eq:Riccati_Z_paper} by separation of variables yields an explicit expression for $Z(t)=U(t)+iP(t)$. In the unit speed case $c=1$ one obtains
\[
Z(t)=-(\theta_1-\theta_2)\tan\!\left(\frac{\theta_1-\theta_2}{2}(t-C)\right)+is
\]
for a constant $C$ determined by the initial condition. Using the tangent addition formula and rewriting the resulting expression in terms of real and imaginary parts gives the representation \eqref{eq:explicit_u_rho_paper}.

For the blow-up criterion, observe that singularities can only occur when the denominator in \eqref{eq:explicit_u_rho_paper} vanishes. This happens if and only if there exist $x_0,t_0$ such that simultaneously
\[
(\rho_0(x_0)-s)\frac{\tan\left(\frac{\theta_1-\theta_2}{2}t_0\right)}{\theta_1-\theta_2}=0,
\qquad
1+(u_0)_x(x_0)\frac{\tan\left(\frac{\theta_1-\theta_2}{2}t_0\right)}{\theta_1-\theta_2}=0.
\]
Since $\delta\neq0$ implies $\theta_1\neq\theta_2$, the second equation enforces $\tan(\frac{\theta_1-\theta_2}{2}t_0)\neq0$, hence $\rho_0(x_0)=s$. Solving the second equation for $t_0$ yields \eqref{eq:blowup_condition_paper}.
\end{proof}

\subsection{Asymptotic absence of blow-up for strong magnetic field}
We conclude by studying the blow-up mechanism as the magnetic strength $s\to\infty$.

\begin{lemma}\label{lem:large_s_no_blowup}
Under the assumptions of Theorem~\ref{thm:explicit_formula}, the expressions \eqref{eq:explicit_u_rho_paper} admit no singularities in the limit $s\to\infty$.
\end{lemma}

\begin{proof}
From \eqref{eq:theta_def_paper} we have
\[
\theta_1-\theta_2=\sqrt{s^2+4(1-s\delta)}.
\]
By Theorem~\ref{thm:explicit_formula}, blow-up requires the existence of $x_0\in\SS^1$ such that $\rho_0(x_0)=s$. Since $\rho_0$ is fixed while $s\to\infty$, no such $x_0$ exists for sufficiently large $s$. Consequently, the denominator in \eqref{eq:explicit_u_rho_paper} cannot vanish, and no singularity occurs.
\end{proof}

\begin{question}
Does there exist $C>0$ such that for all $s>C$ the expressions \eqref{eq:explicit_u_rho_paper} remain globally regular (i.e.\ never develop singularities) for all admissible initial data?
\end{question}

\section{Global conservative weak solutions of the magnetic two-component Hunter--Saxton system}\label{sec:global_weak}

\subsection*{Purpose and structure}
This section provides the analytic details underlying the global-in-time weak theory for the magnetic two-component Hunter--Saxton system (M2HS). The main goals are:
(i) to introduce a natural relaxation $\MAC$ of the configuration space $G^m$ that accommodates the collapse mechanism $\varphi_x\to 0$,
(ii) to formulate an appropriate notion of \emph{weak magnetic geodesic} on $\MAC$ and prove global existence of such flows,
and (iii) to transfer the flow back to Eulerian variables and thereby obtain global conservative weak solutions of (M2HS).

\subsection{Preparation: the closure of $G^m$ and its tangent structure}\label{subsec:closure_Gm}

Recall that for $m>\frac52$, each sufficiently regular solution $(u,\rho)$ of (M2HS) corresponds to a magnetic geodesic $(\varphi,\tau)$ on $(G^m,\GH,\varPhi^*\d\alpha)$ via
\[
u=\varphi_t\circ\varphi^{-1},\qquad \rho=\tau_t\circ\varphi^{-1}.
\]
Moreover, blow-up in Eulerian variables is detected by degeneration of the Lagrangian map: there exist $T>0$ and $x_0\in\SS^1$ such that
\[
\lim_{t\uparrow T}\varphi_x(t,x_0)=0.
\]
In particular, $(\varphi(t),\tau(t))$ may leave $G^m$ at finite time. This motivates relaxing the diffeomorphism component to monotone absolutely continuous maps.

We briefly recall that a function $f:[0,1]\to[0,1]$ is \emph{absolutely continuous} iff it has an a.e.\ derivative $f'\in L^1([0,1])$ and
\[
f(x)=f(0)+\int_0^x f'(y)\,\d y,\qquad x\in[0,1].
\]
We denote by $AC([0,1],[0,1])$ the class of absolutely continuous maps with values in $[0,1]$. While $AC([0,1])$ and $H^1([0,1])$ agree as sets modulo null sets, the constraint $f([0,1])\subset[0,1]$ is an additional property.

Following \cite{wu11}, we introduce the relaxed configuration space.

\begin{definition}[Relaxed configuration space]\label{def:MAC}
Define
\[
\MAC := M_{AC}^0 \rtimes L^2(\SS^1,\RR),
\]
where $M_{AC}^0$ consists of all nondecreasing absolutely continuous maps $\varphi:[0,1]\to[0,1]$ satisfying $\varphi(0)=0$ and $\varphi(1)=1$. The semidirect product multiplication is the same as for $G^m$ (cf.\ \cite[§2]{Lennels13}).
\end{definition}

\begin{remark}
In view of the criterion $\varphi_x\to 0$, the set $M_{AC}^0$ can be regarded as a natural closure of $\Diff_0(\SS^1)$ (orientation-preserving circle diffeomorphisms fixing $0$).
\end{remark}

Identifying $AC(\SS^1)$ with $H^1(\SS^1)$, the tangent space at $(\id,0)\in\MAC$ is
\[
T_{(\id,0)}\MAC=\Bigl\{u\in H^1(\SS^1,\RR):u(0)=0\Bigr\}\times L^2(\SS^1,\RR).
\]
For general $(\varphi,\tau)\in\MAC$, we define the tangent space via right translation:
\begin{equation}\label{eq:Tangent_MAC_def}
T_{(\varphi,\tau)}\MAC
=
\Bigl\{(u\circ\varphi,\rho\circ\varphi):(u,\rho)\in T_{(\id,0)}\MAC\Bigr\}.
\end{equation}

A useful intrinsic characterization is due to Wunsch \cite[Lem.\ 3.1]{wu11}. For $\varphi\in AC(\SS^1)$ set
\begin{equation}\label{eq:Nphi_def}
N_\varphi:=\{x\in\SS^1:\ \varphi_x(x)\ \text{exists and}\ \varphi_x(x)=0\}.
\end{equation}

\begin{lemma}[{\cite[Lem.\ 3.1]{wu11}}]\label{lem:TMAC_characterization}
Let $(\varphi,\tau)\in\MAC$. Then $(U,F)\in AC(\SS^1)\times L^2(\SS^1,\RR)$ belongs to $T_{(\varphi,\tau)}\MAC$ if and only if the following conditions are satisfied:
\begin{enumerate}[label=(\arabic*)]
\item\label{it:TMAC1} $U(0)=0$ and $U_x=0$ a.e.\ on $N_\varphi$,
\item\label{it:TMAC2} $\displaystyle \int_{\SS^1\setminus N_\varphi}\left(\frac{U_x^2}{\varphi_x}+F^2\varphi_x\right)\,\d x<\infty$.
\end{enumerate}
Moreover, for $(U,F),(V,G)\in T_{(\varphi,\tau)}\MAC$ the (right-invariant) metric is
\begin{equation}\label{eq:metric_MAC}
\langle (U,F),(V,G)\rangle_{(\varphi,\tau)}
=
\frac14\int_{\SS^1\setminus N_\varphi}\left(\frac{U_xV_x}{\varphi_x}+FG\,\varphi_x\right)\,\d x.
\end{equation}
\end{lemma}

\subsection{Weak magnetic geodesics on $\MAC$}\label{subsec:weak_magnetic_flow}

We now formulate the weak analogue of the magnetic geodesic flow on the relaxed space $\MAC$.

\begin{definition}[Weak magnetic geodesic]\label{def:weak_magnetic_geodesic}
A curve \( (\varphi, \tau) \) with initial data \( (u_0, \rho_0) \in T_{(\id, 0)}\MAC \) is called a \emph{weak magnetic geodesic} of \( (\MAC, \GH, \varPhi^* \mathrm{d}\alpha) \) if the following conditions are satisfied:
\begin{enumerate}[label=(\arabic*)]
    \item \label{it:weak_geod:in_MAC}
    The curve \( (\varphi, \tau) \) remains in \( \MAC \) for all \( t \geq 0 \); that is,
    \[
    \bigl[x\mapsto(\varphi(t, x), \tau(t, x))\bigr] \in \MAC \quad \forall t \in [0, \infty) \, .
    \]

    \item \label{it:weak_geod:tangent}
    The time derivative \( (\varphi_t, \tau_t) \) lies in the tangent space \( T\MAC \) for all \( t \geq 0 \); that is,
    \[
    \bigl[x\mapsto(\varphi_t(t, x), \tau_t(t, x))\bigr] \in T_{(\varphi(t, \cdot), \tau(t, \cdot))} \MAC \quad \forall t \in [0, \infty)
    \, .
    \]

    \item \label{it:weak_geod:energy}
    The kinetic energy is constant almost everywhere in time; that is, for almost every \( t \in [0, \infty) \),
    \[
    \left\langle (\varphi_t(t), \tau_t(t)), (\varphi_t(t), \tau_t(t)) \right\rangle_{\dot{H}^1, (\varphi, \tau)} =
    \left\langle (\varphi_t(0), \tau_t(0)), (\varphi_t(0), \tau_t(0)) \right\rangle_{\dot{H}^1, (\id, 0)} \, .
    \]
    Moreover, the contact angle is conserved almost everywhere; that is, for almost every \( t \in [0, \infty) \),
    \[
    \delta := \int_{\SS^1} \rho_0(x)\, \mathrm{d}x = \int_{\SS^1} \tau_t(t, x)\, \mathrm{d}x \, .
    \]

    \item \label{it:weak_geod:equation}
    For almost every \( t \in [0, \infty) \), the curve \( (\varphi, \tau) \) satisfies the magnetic geodesic equation as given in \Cref{Prop: magnetic geodesic equation in G}, corresponding to the magnetic system \( (G^m, \GH, \varPhi^* \mathrm{d}\alpha) \).
\end{enumerate}
We further call \( (\varphi, \tau) \) a \emph{unit speed weak magnetic geodesic} if the initial data satisfies
\[
\left\langle (u_0, \rho_0), (u_0, \rho_0) \right\rangle_{\dot{H}^1} = 1 \, .
\]
\end{definition}

The next theorem provides an explicit global construction of weak magnetic geodesics on $\MAC$. Its proof relies on the fact that the Madelung transform is a magnetomorphism (cf.\ the corresponding result in the smooth category).

\begin{theorem}[Global existence of weak magnetic flow]\label{thm:global_weak_magnetic_flow}
Assume $(u_0,\rho_0)\in T_{(\id,0)}\MAC$ satisfies $\langle (u_0,\rho_0),(u_0,\rho_0)\rangle_{\dot H^1}=1$ and define $\theta_{1/2}$ by
\[
\theta_{1/2}=\frac{s\pm\sqrt{s^2+4(1-s\delta)}}{2}.
\]
For $(t,x)\in[0,\infty)\times\SS^1$ set
\begin{equation}\label{eq:phi_def_global}
\varphi(t,x)
=
\int_0^x
\left|
\frac{1}{\theta_2-\theta_1}\Bigl[
\theta_2 e^{\i \theta_1 t}-\theta_1 e^{\i\theta_2 t}
+\frac{\i}{2}\bigl(u_{0,x}(y)+\i \rho_0(y)\bigr)\bigl(e^{\i\theta_1 t}-e^{\i\theta_2 t}\bigr)
\Bigr]
\right|^2
\,\d y,
\end{equation}
and
\begin{equation}\label{eq:tau_def_global}
\tau(t,x)
=
\rho_0(x)\int_0^t
\chi_{\{\varphi_x(s,\cdot)>0\}}\,\frac{1}{\varphi_x(s,x)}\,\d s.
\end{equation}
Then $(\varphi,\tau)$ is a global unit speed weak magnetic geodesic on $(\MAC,\GH,\varPhi^*\d\alpha)$.
\end{theorem}

\begin{remark}
Setting $s=0$ in Theorem~\ref{thm:global_weak_magnetic_flow} recovers \cite[Thm.\ 4.1]{wu11}. Taking $\rho\equiv s$ yields the corresponding statement for the Hunter--Saxton equation \cite[Thm.\ 4.1]{l07.2}.
\end{remark}

\begin{remark}\label{rem:energy_scaling}
The construction extends verbatim to initial data with $\GH$-energy $c^2>0$, producing a weak magnetic geodesic of energy $c^2$.
\end{remark}

\subsection{Global conservative weak solutions of (M2HS)}\label{subsec:global_weak_solutions}

We now define the Eulerian notion of weak solution used in the sequel and show how it is induced by the weak magnetic flow.

The aim of this subsection is to derive, from the global existence of the weak magnetic flow established in \Cref{thm:global_weak_magnetic_flow}, the global existence of a certain notion of global weak solution to~\eqref{eq:M2HS}. To that end, we introduce the following definition.

\begin{definition}[Global conservative weak solution]\label{def:global_conservative_weak_solution}
A pair \( (u, \rho) \colon [0,\infty) \times \SS^1 \to \RR \) is called a \emph{global conservative weak solution} of~\eqref{eq:M2HS} with initial data \( (u_0, \rho_0) \in H^1(\SS^1, \RR) \times L^2(\SS^1, \RR) \) if and only if the following hold:
\begin{enumerate}[label=(\arabic*)]
    \item \label{it:weak_sol:H1}
    For all \( t \in [0, \infty) \), we have \( x\mapsto u(t, x) \in H^1(\SS^1, \RR) \).
    
    \item \label{it:weak_sol:cont_init}
    The map \( u \colon [0,\infty) \times \SS^1 \to \RR \) is continuous, and the initial conditions are satisfied pointwise and almost everywhere, respectively:
    \[
    u(0, x) = u_0(x) \quad \text{for all } x \in \SS^1, 
    \qquad 
    \rho(0, x) = \rho_0(x) \quad \text{for a.e. } x \in \SS^1.
    \]

    \item \label{it:weak_sol:Linfty}
    The maps \( t \mapsto u_x(t,\cdot) \) and \( t \mapsto \rho(t,\cdot) \) belong to \( L^\infty([0,\infty), L^2(\SS^1, \RR)) \).

    \item \label{it:weak_sol:equation}
    The map \( t \mapsto u(t,\cdot) \) is absolutely continuous from \( [0,\infty) \) into \( L^2(\SS^1, \RR) \), and for almost every \( t \in [0,\infty) \), the pair \( (u(t,\cdot), \rho(t,\cdot)) \) satisfies the equation in \Cref{Prop: magnetic geodesic equation in G} evaluated at \( (\id, 0) \); that is,
    \[
    \begin{pmatrix}
        u_t + u u_x \\
        \rho_t + u \rho_x
    \end{pmatrix}
    =
    \begin{pmatrix}
        -\frac{1}{2} A^{-1} \partial_x (u_x^2 + \rho^2) - s \cdot \int_0^x \left( \rho - \int_{\SS^1} \rho \, \mathrm{d}x \right) \mathrm{d}x \\
        -(\rho u)_x + s u_x
    \end{pmatrix}
    \quad \text{in } L^2(\SS^1, \RR) \times L^2(\SS^1, \RR) .
    \]
\end{enumerate}
\end{definition}

\begin{remark}
For \( s = 0 \), \Cref{def:global_conservative_weak_solution} recovers the notion of a global conservative weak solution to the two-component Hunter--Saxton system, as defined in~\cite[Def.~4.2]{wu11}.
\end{remark}

\begin{remark}
For \( \rho \equiv s \), \Cref{def:global_conservative_weak_solution} reduces to the notion of a global conservative weak solution to the (single-component) Hunter--Saxton equation, as introduced in~\cite{l07.2}.
\end{remark}

\begin{remark}
The term \emph{conservative} in \Cref{def:global_conservative_weak_solution} is inspired by the notion of conservative weak solutions to the Hunter--Saxton equation introduced by Bressan--Constantin in~\cite{BressanConstantin}.
\end{remark}

We now aim to extend the aforementioned duality between magnetic geodesics of the system \( (G^m, \GH, \varPhi^* \mathrm{d}\alpha) \) and solutions of~\eqref{eq:M2HS}—as established in \cite[Thm. 5.1]{M24}—to the weak setting.

\begin{theorem}\label{thm:existence_global_weak_sol_M2HS}
Let \( (\varphi, \tau) \) be a weak unit speed magnetic geodesic in \( \MAC \) with initial data \( (u_0, \rho_0) \in H^1(\SS^1, \RR) \times L^2(\SS^1, \RR) \).\\
Then the pair defined by
\begin{equation}\label{eq:weak_solution_def}
    \begin{pmatrix}
        u(t, \varphi(t,x)) \\
        \rho(t, \varphi(t,x))
    \end{pmatrix}
    :=
    \begin{pmatrix}
        \varphi_t(t,x) \\
        \tau_t(t,x)
    \end{pmatrix}, \quad \text{for } (t,x) \in [0,\infty) \times \SS^1,
\end{equation}
is a global conservative weak solution of~\eqref{eq:M2HS} with initial data \( (u_0, \rho_0) \).

Furthermore, the energy is an integral of motion almost everywhere in time; that is, for almost every \( t \in [0, \infty) \),
\begin{equation}\label{eq:energy_integral_of_motion}
    \left\langle (u(t), \rho(t)), (u(t), \rho(t)) \right\rangle_{\dot{H}^1}
    =
    \left\langle (u_0, \rho_0), (u_0, \rho_0) \right\rangle_{\dot{H}^1}
    = 1\, .
\end{equation}
In addition, the contact angle is conserved almost everywhere; that is, for almost all \( t \in [0, \infty) \),
\[
\delta = \int_{\SS^1} \rho_0(x)\, \mathrm{d}x = \int_{\SS^1} \tau_t(t,x)\, \mathrm{d}x \, .
\]
\end{theorem}

\begin{remark}
Setting \( s = 0 \) in \Cref{thm:existence_global_weak_sol_M2HS} recovers \cite[Thm.~4.2]{wu11}.
\end{remark}

\begin{remark}
Taking \( \rho \equiv s \) in \Cref{thm:existence_global_weak_sol_M2HS} recovers \cite[Thm.~4.2]{l07.2}.
\end{remark}

\begin{remark}\label{rem:geometry_varphi_MAC}
See \cite{M24} for a geometric interpretation of \Cref{thm:existence_global_weak_sol_M2HS}, as well as a geometric perspective on its proof. In particular, refer to \cite{M24}.  \Cref{thm:existence_global_weak_sol_M2HS} provides the analytical details of \cite{M24}
\end{remark}

\begin{remark}\label{rem:energy_scaling_weak_geodesics}
\emph{Mutatis mutandis}, the proof of \Cref{thm:existence_global_weak_sol_M2HS} extends to the case where the \( \GH \)-norm of the initial data \( (u_0, \rho_0) \) equals \( c^2 > 0 \). In this case, the resulting curve \( (u, \rho) \) is a conservative global weak solution of \eqref{eq:M2HS} with \( \GH \)-energy \( c^2 \).
\end{remark}
\subsection{Proof of \Cref{thm:global_weak_magnetic_flow}}\label{subsec:proof_global_weak_magnetic_flow}

We prove that the curve $(\varphi,\tau)$ defined by \eqref{eq:phi_def_global} and \eqref{eq:tau_def_global}
is a weak magnetic geodesic in the sense of \Cref{def:weak_magnetic_geodesic}.
The proof is divided into four steps corresponding to conditions
\ref{it:weak_geod:in_MAC}--\ref{it:weak_geod:equation}.

\begin{proof}[Proof of \Cref{thm:global_weak_magnetic_flow}]
\textbf{\ref{it:weak_geod:in_MAC}: $(\varphi(t),\tau(t))\in \MAC$ for all $t\ge0$.}
We first show that $x\mapsto \varphi(t,x)$ belongs to $M_{AC}^0$ for every $t\ge0$.
By the fundamental theorem of calculus we have
\[
\varphi(t,x)=\int_0^x \varphi_x(t,y)\,\d y,
\qquad (t,x)\in[0,\infty)\times \SS^1.
\]
Thus $x\mapsto \varphi(t,x)$ is absolutely continuous for every $t\ge0$ by Lemma~\ref{Technical Lemma 3}.
Moreover, from \eqref{eq:phi_def_global} we directly obtain $\varphi_x(t,x)\ge0$ for a.e.\ $x\in\SS^1$,
hence $\varphi(t,\cdot)$ is nondecreasing, and $\varphi(t,0)=0$.

To verify the endpoint constraint $\varphi(t,1)=1$, define
\begin{equation}\label{eq:gamma_def_paper}
\gamma(t,y):=\frac{1}{\theta_2-\theta_1}\left[
\theta_2 e^{\i \theta_1 t}-\theta_1 e^{\i \theta_2 t}
+\frac{\i}{2}\bigl(u_{0,x}(y)+\i \rho_0(y)\bigr)\bigl(e^{\i\theta_1 t}-e^{\i\theta_2 t}\bigr)\right].
\end{equation}
As explained in the construction, $\gamma$ is a magnetic geodesic on $(\SiL,\GL,\d\alpha)$ and therefore satisfies
\[
\|\gamma(t,\cdot)\|_{L^2(\SS^1)}=1\qquad\forall\,t\ge0.
\]
Since $\varphi_x(t,x)=|\gamma(t,x)|^2$, we infer
\[
\varphi(t,1)=\int_0^1 \varphi_x(t,x)\,\d x
=\int_0^1|\gamma(t,x)|^2\,\d x
=\|\gamma(t,\cdot)\|_{L^2}^2=1.
\]
Hence $\varphi(t,\cdot)\in M_{AC}^0$ for all $t\ge0$.

Finally, by the explicit definition \eqref{eq:tau_def_global} we have
\[
\tau(t,\cdot)\in L^2(\SS^1,\RR)\qquad\forall\,t\ge0.
\]
Thus $(\varphi(t),\tau(t))\in\MAC$ for all $t\ge0$.

\medskip
\textbf{\Cref{it:weak_geod:tangent}: $(\varphi_t(t),\tau_t(t))\in T_{(\varphi(t),\tau(t))}\MAC$ for all $t\ge0$.}
We verify the characterization given in \Cref{lem:TMAC_characterization}.
First we show
\begin{equation}\label{eq:varphit_taut_regularity_paper}
(\varphi_t(t,\cdot),\tau_t(t,\cdot))\in AC(\SS^1)\times L^2(\SS^1)
\qquad\forall\,t\ge0.
\end{equation}
Using $\varphi_x=|\gamma|^2$ and differentiating under the integral sign, we obtain for all $t\ge0$
\begin{equation}\label{eq:varphi_t_formula_paper}
\varphi_t(t,x)=\int_0^x \left(\gamma(t,y)\,\overline{\dot\gamma(t,y)}+\overline{\gamma(t,y)}\,\dot\gamma(t,y)\right)\,\d y,
\qquad x\in\SS^1.
\end{equation}
Hence $\varphi_t(t,\cdot)$ is absolutely continuous by Lemma~\ref{Technical Lemma 3}. Moreover, from \eqref{eq:tau_def_global}
we have $\tau_t(t,\cdot)\in L^2(\SS^1)$.

Next we check condition \ref{it:TMAC1} in \Cref{lem:TMAC_characterization}.
Since $\varphi_t(t,0)=\varphi_t(t,1)$, \eqref{eq:varphi_t_formula_paper} yields
\[
\varphi_t(t,1)=2\langle \gamma(t),\dot\gamma(t)\rangle_{L^2}.
\]
Because $\dot\gamma(t)\in T_{\gamma(t)}\SiL$, it is $L^2$-orthogonal to $\gamma(t)$, hence $\varphi_t(t,1)=0$,
and therefore $\varphi_t(t,0)=0$.

Let $N:=\{x:\varphi_x(t,x)=0\}$. Using the identity
\[
\varphi_{tx}=(u_x\circ\varphi)\,\varphi_x,
\]
we conclude that $\varphi_{tx}=0$ a.e.\ on $N$, which is precisely the degeneracy condition required in \Cref{lem:TMAC_characterization}.

Finally, condition \Cref{it:TMAC2} follows from the energy estimate proved in the Step~ checking \Cref{it:weak_geod:energy} below
(cf.\ \eqref{eq:energy_step_paper}--\eqref{eq:energy_minus_zero_set_paper}). Consequently,
\[
(\varphi_t(t),\tau_t(t))\in T_{(\varphi(t),\tau(t))}\MAC\qquad\forall\,t\ge0.
\]

\medskip
\textbf{\Cref{it:weak_geod:energy}: Conservation of energy and contact angle (a.e.\ in $t$).}
Let $N=\{x:\varphi_x(t,x)=0\}$. The $\dot H^1$-energy on $\SS^1\setminus N$ is
\begin{align}\label{eq:energy_step_paper}
\GH_{(\varphi(t),\tau(t))}\bigl((\varphi_t,\tau_t),(\varphi_t,\tau_t)\bigr)
&=\frac14\int_{\{\varphi_x>0\}}\left(\frac{\varphi_{tx}^2}{\varphi_x}+\tau_t^2\varphi_x\right)\,\d x.
\end{align}
Write $\gamma=f+\i g$ with $f=\Re\gamma$, $g=\Im\gamma$. Then $\varphi_x=f^2+g^2$ and $\{\varphi_x=0\}=\{f=0\}\cap\{g=0\}$.
A computation yields
\begin{equation}\label{eq:energy_identity_paper}
\frac14\int_{\{\varphi_x>0\}}\left(\frac{\varphi_{tx}^2}{\varphi_x}+\tau_t^2\varphi_x\right)\,\d x
=\int_{\{|\gamma|>0\}}|\dot\gamma|^2\,\d x.
\end{equation}
Hence
\begin{equation}\label{eq:energy_decomposition_paper}
\int_{\{|\gamma|>0\}}|\dot\gamma|^2\,\d x
=
\int_{\SS^1}|\dot\gamma|^2\,\d x
-\int_{\{f=0\}\cap\{g=0\}}|\dot\gamma|^2\,\d x.
\end{equation}
Since the initial data are chosen unit speed, $\int_{\SS^1}|\dot\gamma|^2\,\d x=1$, and thus
\begin{equation}\label{eq:energy_minus_zero_set_paper}
\GH_{(\varphi(t),\tau(t))}\bigl((\varphi_t,\tau_t),(\varphi_t,\tau_t)\bigr)
=
1-\int_{\{f=0\}\cap\{g=0\}}|\dot\gamma|^2\,\d x.
\end{equation}
It remains to show that the latter integral in \eqref{eq:energy_minus_zero_set_paper} vanishes for a.e.\ $t\ge0$.
This follows from the fact that for each fixed $x_0\in\SS^1$, the function $t\mapsto\gamma(t,x_0)$ is a finite linear combination
of $e^{\i\theta_1 t}$ and $e^{\i\theta_2 t}$, hence has only finitely many zeros on bounded intervals.
Applying Fubini’s theorem implies that the zero set has measure zero in space for a.e.\ time. Therefore the energy is conserved for a.e.\ $t$.

The contact angle identity
\[
\delta=\int_{\SS^1}\tau_t(t,x)\,\d x=\int_{\SS^1}\rho_0(x)\,\d x
\]
is immediate from the definition of $\tau$ in \eqref{eq:tau_def_global}.

\medskip
\textbf{\Cref{it:weak_geod:equation}: Weak magnetic geodesic equation.}
We show that $(\varphi,\tau)$ satisfies the magnetic geodesic equation for a.e.\ $t\ge0$.
Differentiating $\varphi$ twice and using $\varphi_x=f^2+g^2$ yields
\begin{align}\label{eq:varphi_tt_start_paper}
\varphi_{tt}(t,x)
&=2\int_0^x\bigl(f_t^2+g_t^2+f_{tt}f+g_{tt}g\bigr)\,\d y.
\end{align}
Since $\gamma$ is a magnetic geodesic on $(\SiL,\GL,\d\alpha)$, it satisfies
\[
\ddot\gamma-s\i \dot\gamma+(1-s\delta)\gamma=0.
\]
Equivalently,
\[
f_{tt}+s g_t+(1-s\delta)f=0,
\qquad
g_{tt}-s f_t+(1-s\delta)g=0.
\]
Substituting into \eqref{eq:varphi_tt_start_paper} and using the identities
\[
u=\varphi_t\circ\varphi^{-1},\qquad \rho=\tau_t\circ\varphi^{-1},
\]
together with the transformation rule for monotone absolutely continuous maps
(Lemma~\ref{Technical lemma 2/ trafo rule for abs cont funct}),
gives the weak formulation
\begin{equation}\label{eq:varphi_tt_weak_form_paper}
\varphi_{tt}(t,x)
=
\frac12\int_0^{\varphi(t,x)} (u_x^2+\rho^2)\,\d y
-s\int_0^{\varphi(t,x)}\rho\,\d y
-2\varphi(t,x)(1-s\delta).
\end{equation}
Comparing \eqref{eq:varphi_tt_weak_form_paper} with Proposition~\ref{Prop: magnetic geodesic equation in G}
yields the first component of the magnetic geodesic equation a.e.\ in time.

For the second component, differentiate the identity
\[
\tau_t = 2\,\frac{g_t f-f_t g}{f^2+g^2}
\]
and use the ODE for $(f,g)$. This yields
\[
\tau_{tt}=s\frac{\varphi_{xt}}{\varphi_x}-\tau_t\,\frac{\varphi_{xt}}{\varphi_x}.
\]
Since $\varphi_{xt}=(u_x\circ\varphi)\varphi_x$ and $\tau_t=\rho\circ\varphi$, we obtain
\[
\rho_t\circ\varphi = (s u_x-(\rho u)_x)\circ\varphi,
\]
which is exactly the second component of the magnetic geodesic equation. This completes the proof of \ref{it:weak_geod:equation}.

Combining Steps~1--4 shows that $(\varphi,\tau)$ is a weak magnetic geodesic on $\MAC$.
\end{proof}

\subsection{Proof of \Cref{thm:existence_global_weak_sol_M2HS}}\label{subsec:proof_global_weak_solutions}

Let $(\varphi,\tau)$ be a weak magnetic geodesic on \texorpdfstring{$\MAC$}{MAC} with initial data $(u_0,\rho_0)$, and define $(u,\rho)$ by
\[
u(t,\varphi(t,x))=\varphi_t(t,x),\qquad \rho(t,\varphi(t,x))=\tau_t(t,x).
\]
We show that $(u,\rho)$ is a global conservative weak solution of \eqref{eq:M2HS} in the sense of \Cref{def:global_conservative_weak_solution}.
The argument is organized into four steps corresponding to items
\ref{it:weak_sol:H1}--\ref{it:weak_sol:equation} of \Cref{def:global_conservative_weak_solution}.

\begin{proof}[Proof of \Cref{thm:existence_global_weak_sol_M2HS}]
\textbf{\Cref{it:weak_sol:H1}: $u(t,\cdot)\in H^1(\SS^1)$ for all $t\ge0$, and energy conservation a.e.\ in $t$.}
By definition of the tangent bundle of $\MAC$ via right translation (cf.\ \eqref{def:MAC}) and since
$(\varphi_t(t),\tau_t(t))\in T_{(\varphi(t),\tau(t))}\MAC$ for all $t\ge0$
(\Cref{def:weak_magnetic_geodesic}\,\ref{it:weak_geod:in_MAC}), we have
\[
(\varphi_t(t),\tau_t(t))=(u(t)\circ\varphi(t),\rho(t)\circ\varphi(t))
\]
with $(u(t),\rho(t))\in T_{(\id,0)}\MAC= \{v\in H^1(\SS^1):v(0)=0\}\times L^2(\SS^1)$.
In particular, $u(t,\cdot)\in H^1(\SS^1)$ for every $t\ge0$.

Next, using $u=\varphi_t\circ\varphi^{-1}$ and $\rho=\tau_t\circ\varphi^{-1}$ together with right-invariance of $\GH$
and the change-of-variables rule for absolutely continuous monotone maps
(Lemma~\ref{Technical lemma 2/ trafo rule for abs cont funct}), we obtain
\begin{align}
\GH_{(\id,0)}((u(t),\rho(t)),(u(t),\rho(t)))
&=\frac14\int_{\SS^1}\left(u_x(t,x)^2+\rho(t,x)^2\right)\,\d x \nonumber\\
&=\frac14\int_{\SS^1}\left(\frac{\varphi_{tx}(t,x)^2}{\varphi_x(t,x)}+\tau_t(t,x)^2\,\varphi_x(t,x)\right)\,\d x \nonumber\\
&=\GH_{(\varphi(t),\tau(t))}\bigl((\varphi_t(t),\tau_t(t)),(\varphi_t(t),\tau_t(t))\bigr).
\label{eq:energy_pushforward_paper}
\end{align}
Since $(\varphi,\tau)$ is a weak magnetic geodesic, its $\GH$-energy is conserved for a.e.\ $t$ by
\Cref{def:weak_magnetic_geodesic}\,\ref{it:weak_geod:energy}.
Therefore, for a.e.\ $t\ge0$,
\[
\GH_{(\id,0)}((u(t),\rho(t)),(u(t),\rho(t)))
=\GH_{(\id,0)}((u_0,\rho_0),(u_0,\rho_0)),
\]
which is exactly \eqref{eq:energy_integral_of_motion}.

\medskip
\textbf{\Cref{it:weak_sol:cont_init}: continuity of $u$ and attainment of initial data.}
We show that $u:[0,\infty)\times\SS^1\to\RR$ is continuous and that $u(0,\cdot)=u_0$ pointwise, while $\rho(0,\cdot)=\rho_0$ a.e.

Let $t_n\to t$ and $y_n\to y$ in $\SS^1$. Since for each fixed $t$ the map $x\mapsto\varphi(t,x)$ is nondecreasing and surjective,
we may choose $x_n\in\SS^1$ with
\[
\varphi(t_n,x_n)=y_n.
\]
Then, by definition,
\[
u(t_n,y_n)=u(t_n,\varphi(t_n,x_n))=\varphi_t(t_n,x_n).
\]
Any limit point of $(x_n)$ lies in the closed level set $\varphi(t,\cdot)^{-1}(\{y\})$.
Using the continuity of $(t,x)\mapsto\varphi_t(t,x)$ and the fact that $\varphi_t(t,\cdot)$ is constant on each level set
$\varphi(t,\cdot)^{-1}(\{y\})$, we conclude that $u(t_n,y_n)\to u(t,y)$, i.e.\ $u$ is continuous.

Moreover, since $\varphi(0,x)=x$ and $\varphi_t(0,x)=u_0(x)$, we have $u(0,x)=u_0(x)$ for all $x$.
For $\rho$, using $\rho\circ\varphi=\tau_t$ and the definition of $\tau$ in \Cref{thm:global_weak_magnetic_flow}, we compute for a.e.\ $x$, 
\[
\rho(0,x)=\tau_t(0,x)
=\rho_0(x)\,\left.\frac{1}{\varphi_x(t,x)}\right|_{t=0}
=\rho_0(x),
\]
since $\varphi_x(0,x)=1$. This proves \ref{it:weak_sol:cont_init}.

\medskip
\textbf{\Cref{it:weak_sol:Linfty}: $u_x,\rho \in L^\infty([0,\infty);L^2(\SS^1))$.}
The energy identity \eqref{eq:energy_pushforward_paper} implies that $\|u_x(t)\|_{L^2}^2+\|\rho(t)\|_{L^2}^2$
is uniformly bounded in $t$ (a.e.).
Thus it remains to justify measurability of the maps
\[
t\mapsto u_x(t,\cdot)\in L^2(\SS^1),\qquad t\mapsto \rho(t,\cdot)\in L^2(\SS^1),
\]
so that they represent elements of $L^\infty([0,\infty);L^2)$.

For $u_x$, by Pettis' theorem (Theorem~\ref{Pettis thm}) it suffices to show that for every $\Theta\in L^2(\SS^1)$ the scalar map
\begin{equation}\label{eq:scalar_pairing_ux}
t\longmapsto \int_{\SS^1} u_x(t,x)\,\Theta(x)\,\d x
\end{equation}
is measurable, and it is enough to prove continuity when $\Theta$ is smooth.

Assume first $\Theta\in C^\infty(\SS^1)$. Using the change-of-variables formula,
\[
\int_{\SS^1} u_x(t,x)\Theta(x)\,\d x
=\int_{\SS^1} u_x(t,\varphi(t,x))\,\Theta(\varphi(t,x))\,\varphi_x(t,x)\,\d x.
\]
Since $\varphi_{xt}=(u_x\circ\varphi)\varphi_x$, we obtain
\[
\int_{\SS^1} u_x(t,x)\Theta(x)\,\d x
=\int_{\SS^1} \varphi_{xt}(t,x)\,\Theta(\varphi(t,x))\,\d x.
\]
In the notation of the proof of \Cref{thm:global_weak_magnetic_flow}, one has $\varphi_{xt}=2(ff_t+gg_t)$ with $\gamma=f+\i g$, hence
\begin{equation}\label{eq:pairing_ux_gamma}
\int_{\SS^1} u_x(t,x)\Theta(x)\,\d x
=2\int_{\SS^1}\bigl(f(t,x)f_t(t,x)+g(t,x)g_t(t,x)\bigr)\,\Theta(\varphi(t,x))\,\d x.
\end{equation}
The right-hand side depends continuously on $t$ because $t\mapsto (f,g,\varphi)$ is continuous. Hence \eqref{eq:scalar_pairing_ux} is continuous for smooth $\Theta$.

For general $\Theta\in L^2$, choose $\Theta_n\in C^\infty$ with $\Theta_n\to \Theta$ in $L^2$.
For each fixed $t$, the functional $\Theta\mapsto \int u_x(t)\Theta$ is continuous on $L^2$; thus
\[
\int_{\SS^1} u_x(t)\Theta_n \,\d x \to \int_{\SS^1} u_x(t)\Theta\,\d x.
\]
Therefore \eqref{eq:scalar_pairing_ux} is a pointwise limit of continuous maps, hence measurable.
This proves $t\mapsto u_x(t)\in L^\infty([0,\infty);L^2)$.

The corresponding statement for $t\mapsto \rho(t)\in L^\infty([0,\infty);L^2)$ follows analogously
(cf.\ the argument in \cite[Thm.~4.2]{wu11}), using $\rho\circ\varphi=\tau_t$ and the $L^2$-regularity of $\tau_t$.

\medskip
\textbf{\Cref{it:weak_sol:equation}: absolute continuity in $L^2$ and the evolution equations in $L^2\times L^2$.}
We begin by stating two claims, from which we derive the desired result, and subsequently provide the proofs of the claims.

\smallskip
\begin{clai}
\label{claim_1}For any test functions \( \eta \in C_c^{\infty}((0,\infty)) \) and \( \Theta \in C^{\infty}(\SS^1) \), it holds that 
\[
\int_{\SS^1}\!\bigg[
\int_0^\infty u(t,x)\eta_t(t)\,\d t
+\int_0^\infty\Bigl(\Gamma_{(\id,0)}^{(1)}(u,\rho)(u,\rho)
+sY_{(\id,0)}^{(1)}(u,\rho)-u u_x\Bigr)\eta(t)\,\d t
\bigg]\Theta(x)\,\d x=0.
\]
\end{clai}
\begin{clai}\label{claim_2}
    For any test functions \( \eta \in C_c^{\infty}((0,\infty)) \) and \( \Theta \in C^{\infty}(\SS^1) \), it holds that
\[
\int_{\SS^1}\!\bigg[
\int_0^\infty \rho(t,x)\eta_t(t)\,\d t
+\int_0^\infty\Bigl(\Gamma_{(\id,0)}^{(2)}(u,\rho)(u,\rho)
+sY_{(\id,0)}^{(2)}(u,\rho)-u \rho_x\Bigr)\eta(t)\,\d t
\bigg]\Theta(x)\,\d x=0.
\]
\end{clai} 

Assuming Claims~\ref{claim_1}-\ref{claim_2}, Lemma~\ref{Technical Lemma 4} implies that
\[
t\mapsto u(t)\in L^2(\SS^1)\quad\text{and}\quad t\mapsto \rho(t)\in L^2(\SS^1)
\]
are absolutely continuous, and for a.e.\ $t$ we have in $L^2(\SS^1)$
\[
u_t = \Gamma_{(\id,0)}^{(1)}(u,\rho)(u,\rho)+sY_{(\id,0)}^{(1)}(u,\rho)-u u_x,
\qquad
\rho_t = \Gamma_{(\id,0)}^{(2)}(u,\rho)(u,\rho)+sY_{(\id,0)}^{(2)}(u,\rho)-u \rho_x,
\]
which is exactly \Cref{def:global_conservative_weak_solution}\,\ref{it:weak_sol:equation}.
It remains to prove the two claims.

\medskip\noindent
\textbf{Proof of \Cref{claim_1}.}
By Step~\ref{it:weak_sol:Linfty} and Proposition~\ref{Prop: magnetic geodesic equation in G}, the maps
\[
t\mapsto u(t,\cdot),\qquad
t\mapsto \Gamma_{(\id,0)}^{(1)}(u,\rho)(u,\rho)+sY_{(\id,0)}^{(1)}(u,\rho)-u u_x
\]
belong to $L^\infty([0,\infty);L^2(\SS^1))$, so Fubini’s theorem applies.

Using $u\circ\varphi=\varphi_t$ and the transformation rule
(Lemma~\ref{Technical lemma 2/ trafo rule for abs cont funct}),
\[
\int_{\SS^1} u(t,x)\Theta(x)\,\d x
=\int_{\SS^1} \varphi_t(t,x)\,(\Theta\circ\varphi(t,x))\,\varphi_x(t,x)\,\d x.
\]
Integrating by parts in $t$ against $\eta$ yields
\[
\int_0^\infty\!\int_{\SS^1} u(t,x)\Theta(x)\,\d x\,\eta_t(t)\,\d t
=
-\int_0^\infty\!\int_{\SS^1}\frac{\d}{\d t}\Bigl(\varphi_t(\Theta\circ\varphi)\varphi_x\Bigr)\,\d x\,\eta(t)\,\d t.
\]
Expanding the time derivative gives
\[
-\int_0^\infty\!\int_{\SS^1}\bigl(
\varphi_{tt}(\Theta\circ\varphi)\varphi_x
+\varphi_t^2(\Theta_x\circ\varphi)\varphi_x
+\varphi_t(\Theta\circ\varphi)\varphi_{xt}
\bigr)\,\d x\,\eta(t)\,\d t.
\]
Since $(\varphi,\tau)$ is a weak magnetic geodesic, for a.e.\ $t$ we have
\[
\varphi_{tt}=\bigl(\Gamma_{(\id,0)}^{(1)}(u,\rho)(u,\rho)+sY_{(\id,0)}^{(1)}(u,\rho)\bigr)\circ\varphi,
\]
and moreover $\varphi_t=u\circ\varphi$ and $\varphi_{xt}=(u_x\circ\varphi)\varphi_x$.
Applying again the transformation rule and integrating by parts in $x$ in the $u^2\Theta_x$ term yields exactly \Cref{claim_1}.

\medskip\noindent
\textbf{Proof of \Cref{claim_2}.}
As in Claim~1, Step~\ref{it:weak_sol:Linfty} implies the relevant $L^\infty_tL^2_x$ bounds.
Using $\rho\circ\varphi=\tau_t$ and Lemma~\ref{Technical lemma 2/ trafo rule for abs cont funct},
\[
\int_{\SS^1}\rho(t,x)\Theta(x)\,\d x
=\int_{\SS^1}\tau_t(t,x)\,(\Theta\circ\varphi(t,x))\,\varphi_x(t,x)\,\d x.
\]
Integrating by parts in $t$ against $\eta$ gives
\[
\int_0^\infty\!\int_{\SS^1}\rho(t,x)\Theta(x)\,\d x\,\eta_t(t)\,\d t
=
-\int_0^\infty\!\int_{\SS^1}\bigl(
\tau_{tt}(\Theta\circ\varphi)\varphi_x
+\tau_t(\Theta_x\circ\varphi)\varphi_t\varphi_x
+\tau_t(\Theta\circ\varphi)\varphi_{tx}
\bigr)\,\d x\,\eta(t)\,\d t.
\]
Since $(\varphi,\tau)$ solves the magnetic geodesic equation in $\MAC$, we have for a.e.\ $t$
\[
\tau_{tt}=\bigl(\Gamma_{(\id,0)}^{(2)}(u,\rho)(u,\rho)+sY_{(\id,0)}^{(2)}(u,\rho)\bigr)\circ\varphi,
\]
and using $\tau_t=\rho\circ\varphi$, $\varphi_t=u\circ\varphi$ and $\varphi_{tx}=(u_x\circ\varphi)\varphi_x$,
another application of the transformation rule yields \Cref{claim_2}. This proves Claim~2.

With Claims~1--2 established, Step~\ref{it:weak_sol:equation} follows, and hence $(u,\rho)$ is a global conservative weak solution of \eqref{eq:M2HS}.
\end{proof}


\appendix


\section{Technical Lemmata}\label{Appendix global weak solt: techni lemmas}
The following three lemmas can be found in \cite{mp84} and \cite{r87}
\begin{lemma}\label{Technical Lemma 1} Let $f:[0,1]\longrightarrow \RR$ be a absolutely continuous function. Then $f$ maps measurable sets into measurable sets and if $f_x$ exists and equals $0$ a.e. on a measurable set $N\subset [0,1]$ then $\lambda (f(N))=0$, where $\lambda$ denotes the Lebesgue measure.
\end{lemma} 
See \cite[Thm 7.20]{r87} for a proof of the following lemma
\begin{lemma}\label{Technical lemma 2/ trafo rule for abs cont funct}
    Suppose $\varphi:[0,1]\longrightarrow [0,1]$ is absolutely continuous, non decreasing, $\varphi(0)=0, \varphi(1)=1$ and $f\in L^1\left([0,1]\right)$. Then 
    \[
    \int_0^1f(y)\,  \d y = \int_0^1f(\varphi(x))\, \varphi_x(x)\,  \d x  \, .\]
    The same conclusion holds if $f$ is measurable and $f\geq 0$.
\end{lemma}
See \cite[Thm 7.18]{r87} for a proof of the following lemma
\begin{lemma}\label{Technical Lemma 3}
    Let $f:[a,b]\longrightarrow  \RR$ be continuous and non decreasing. Each of the following three statements implies the other two
    \begin{enumerate}[label=(\arabic*)]
        \item The function $f$ is absolutely continuous. 
        \item The function $f$ maps sets of measure $0$ to sets of measure $0$.
        \item The function $f$ is differentiable a.e. on $I$, $f_x\in L^1([a,b])$ and for all $x\in I$ it holds that
        \[
        f(x)=f(a)+\int_a^xf_x(y)\d y\text{.}
        \]
    \end{enumerate}
\end{lemma}
See \cite[Chapter 3, Lemma 1.1.]{t84} for a proof of the following lemma
\begin{lemma}\label{Technical Lemma 4}
    Let $X$ be a Banach space with dual $X^{\prime}$, and let $u,g\in L^1\left([0,T),X\right)$ for some $T>0$. Then the following two conditions are equivalent\begin{enumerate}[label=(\arabic*)]
        \item\label{it: 1 tech Lemm 4} The function $u$ is a.e. equal to a primitive function of $g$ i.e. there exists a $F\in X$ such that \[
u(t)= F+\int_{0}^t g(s)\d s\quad \text{for a.e.}\, \, t\in[0,T)\, .
        \]
        \item\label{it: 2 tech Lemm 4}  For each test function $\eta\in C_c^{\infty}((0,T))$ it holds that:
        \[
        \int_{0}^T u(t)\, \eta_t(t)\, \d t=-\int_0^T g(t)\, \eta(t)\, \d t\, .
        \]
    \end{enumerate}
    If \ref{it: 1 tech Lemm 4} or \ref{it: 2 tech Lemm 4} are satisfied, then $u$ is a.e. equal to an absolutely continuous function \[v:[0,T)\longrightarrow X\, .\]
\end{lemma}

For the sake of completeness we will state Pettis measurability theorem:
\begin{theorem}\label{Pettis thm}
Let $(X,\sigma,\mu)$ be a measure space and let $B$ be a Banach space. Then \begin{enumerate}[label=(\arabic*)]
    \item The function $f:X\longrightarrow B$ is measurable if and only if it is weakly measurable and almost surely separably valued.
    \item If $B$ is separable Banach space then  a function $f:X\longrightarrow B$ is measurable if and only if it is weakly measurable.
\end{enumerate}  
\end{theorem}


\bibliographystyle{abbrv}
 \bibliography{ref}

\end{document}